\documentclass[review,16pt]{elsarticle}
\usepackage[utf8]{inputenc}
 \usepackage{amsmath}
 \usepackage{amssymb}
 \usepackage{amsthm}
\usepackage{lineno,hyperref}
\DeclareUnicodeCharacter{2212}{-}

 \newtheorem{theorem}{Theorem}[section]
\newtheorem{lemma}[theorem]{Lemma}

\theoremstyle{definition}
\newtheorem{definition}[theorem]{Definition}
\newtheorem{example}[theorem]{Example}

\newtheorem{proposition}[theorem]{Proposition}
\theoremstyle{remark}
\newtheorem{remark}[theorem]{Remark}
\newtheorem{corollary}[theorem]{Corollary}

\begin{document}

\begin{frontmatter}

\title{Non-stationary $\phi$-contractions and associated fractals}

\author[mymainaddress]{Amit}

\ead{amit60363@gmail.com}

\author[mymainaddress]{Vineeta Basotia}
\ead{vm.jjt@gmail.com}

\author[secondarymainaddress]{Ajay Prajapati}
\ead{ajaypraja640@gmail.com}

\address[mymainaddress]{Department of Mathematics, JJT University, Jhunjhunu, Rajasthan,  India, 333001}

\address[secondarymainaddress]{Department of Computer Science, Banaras Hindu University, Varansi, India, 221005}

\begin{abstract}
In this study we provide several  significant generalisations of Banach contraction principle where the Lipschitz constant is substituted by real-valued control function that is a comparison function. We study non-stationary variants of fixed-point. In particular, this article looks into “trajectories of maps defined by function systems” which are regarded as generalizations of traditional iterated function system. The importance of forward and backward trajectories of general sequences of mappings is analyzed. The convergence characteristics of these trajectories determined a non-stationary variant of the traditional fixed-point theory. Unlike the normal fractals which have self-similarity at various scales, the attractors of these  trajectories of maps which defined by function  
systems that may have various structures at various scales. In this literature we also study the sequence of countable IFS having some generalized contractions on a complete metric space
 
\end{abstract}

\begin{keyword}
\texttt{Forward Trajectories \sep Backward Trajectories \sep Comparison Function \sep Iterated Function Systems \sep Fractal interpolation}
\end{keyword}

\end{frontmatter}


\section{Introduction}
The most renowned contribution in the study of fixed points theory is the Banach contraction principle.
In \cite{SR1} we see some valuable  generalizations of Banach contraction principle where the Lipschitz's constant is shifted by some control function with real-values. In complete metric space, Boyd and Wong \cite{PV1} revealed the existence of fixed points of mappings and also noted the Banach contraction principle that have some of generalized contractions condition.
Dyn et al. \cite{DLM,DLM2} studied the non-stationary fractals. It should also be noted that Graf \cite{Graf} had already studied this problem under the title of statistically self-similar sets.

The purpose of this article is to getting the  theorems of fixed point by generalizing the contractions or also introducing  the trajectories which induced with sequence of many countable transformations or also establish its elementary properties (see \cite{DLM2}). The concept of fractal interpolation function(FIF) was introduced by Barnsley\cite{MB} through iterated function system (IFS). An IFS have contractions mapping that belongs to finite family over some complete metric space $X\to X$. An IFS induces an operator \cite{Secelean} on $K(X)$(collection of all nonempty compact subsets of $X$) into itself which is a contraction on this complete metric space hence, according to Banach’s Contraction Principle, has a unique set “fixed-point” ( also called attractor of the respective IFS) which is, generally, a fractal \cite{sec}. Generalized countable IFS (GCIFS) are expansion of countable IFS rather to define contraction on complete metric space $X$, contraction define on $X\times X \to X$ (see \cite{sec22}).\cite{storbin} examined that there exist a subset of plane which is an attractor of generalized IFS, however its not an attractor of IFS also find out an examples of  Cantors set that is not an attractor of  generalized IFS.
\par 
We first explain some fundamental results before establishing the fixed point theorems of any generalised contraction. Fractal dimension is an integral part of fractal geometry. Many recent works on fractal dimensions can be seen in the literature. 
Liang \cite{Liang}  confirmed that any function which is continuous and of bounded variation on $[0,1]$ having one box dimension. He also confirmed that the Riemann Liouville fractional integral of any continuous function which is defined on bounded variation on $[0,1]$ have one box dimension. The effect of the Riemann-Liouville fractional integral operator on unbounded variation points of a continuous function can be seen \cite{VL1}. When we move from unit interval to rectangular domain, we have different ideas of bounded variations such as Arzel\'{a}, Toneli, Hahn and Peirpont, for instance, see \cite{JA}. By using the concept of bounded variation with respect to Arzel\'{a},  Verma and Viswanathan gave results on the fractal dimension of the graph of the mixed Riemann-Liouville fractional integrals over rectangular domain in \cite{sverma3}. In this order, Chandra and Abbas \cite{SS2} estimated the fractal dimension of the graph of the mixed Riemann-Liouville fractional integrals of different choice of  continuous functions over rectangular domain.   They also estimated the box dimension of the graph of the mixed Katugampola fractional integrals of a two dimensional continuous functions in \cite{SS3}. The fractal dimensional results of the mixed Weyl-Marchaud fractional derivative can be seen in \cite{SS4}.  Chandra and Abbas\cite{SS1} examined the partial integrals as well as  partial derivatives about the bivariate fractal interpolation functions (FIFs) and also study  the integral transforms and fractional order of integral transforms of bivariate FIFs. Gowrisankar and Uthayakumar \cite{Gowri} studied about the presence of fractal interpolation function (FIF) for the data  sequence given as   $\{(a_n,b_n):n\geq 2\}$ having countable iterated function system where the sequence  $a_n$ is bounded as well as monotone and sequence $b_n$ is bounded, also the order integral of FIF for sequence data is evaluate if the values at initial endpoint and final endpoint of the integral is known.\\
Jha and Verma\cite{Jha} provided  a hard study of fractal dimension of  $\alpha-$fractal functions on various function spaces. Jha \cite{JVC1} shown the method proposed generalize the existing stationary interpolant in terms of IFS and also worked on the proposed interpolant's basic properties and figured out its box dimension as well as Hausdorff dimension. Sahu and Priyadrshi estimated the box dimension of Graph of harmonic function on Sierpi\'nski  gasket in \cite{SP}. Verma and Sahu \cite{VFS} introduced the  notion of bounded variation on  Sierpi\'nski  gasket. Also, they have estimated the fractal dimension of a continuous function which is  of bounded vartion on  Sierpi\'nski  gasket. In this order, Agrawal and Som  \cite{Vishal} have presented the concept of dimension preserving approximation for the bivariate continuous function on a rectangular domain. Agrawal and Som \cite{EPJST} have estimated fractal dimensions of $\alpha$-fractal functions on the Sierpi\'nski Gasket. Further, the same authors \cite{RIM} have investigated the approximation of functions by fractal functions corresponding to the $L^p$-norm on the Sierpi\'nski gasket. New idea for construction of FIF by using  Rakotch and Matkowski fixed point theorems is given in \cite{SR2}. Verma et al. \cite{Ver1} initiated theory on fractal dimension of vector valued function.
\par
Mihail \cite{redu} proved the existence of an analogue of Hutchinson measure associated with a GIFS with probabilities and present some of its properties. Pandey et al.  \cite{PSV1} developed some theory about the dimension of vector valued function graph also figured out the fractal dimension of the graphs of this Katugampola fractional integral about  vector-valued continuous functions having bounded variation which defined on closed and bounded interval of $\mathbb{R}$. Verma \cite{VM1} investigated approximation issues in relation to the fractal dimensions of function as well as their derivative.
\par
The organization of this paper is as follows. In Section 2, we discuss about basics definition and examples relates to our result that we generalize in this paper. In Section 3, we discuss about some contractive condition and its example and properties related to contraction i.e convergence of forward as well as backward trajectories. Further, Section 4 is about the trajectories of sequence of function systems (SFSs) asymptotic similarity of SFS trajectories and convergence of forward and backward SFS trajectories and some result based on $\phi-$ contraction. In Section 5, we discuss about the countable iterated function system consisting of $\phi-$ contractions.

\section{Preliminary facts}
We will revisit several well-known principles of fixed point theory in this part, that will be included in the follow.
\begin{definition}\label{con}(Contraction mapping)
Assume $X$ is a metric space with  metric $d$. A function $f : X \to X$ is Known as Lipschitz continuous if $\exists ~\lambda \geq 0$ in the sense that
$$d(f(a), f(b)) \leq \lambda d(a,b), ~\forall~ a, b \in X.$$
The Lipschitz constant of $f$ is the lowest $\lambda$ for which the preceding inequality holds. If $\lambda < 1$ is true, then $f$ is a contraction mapping.
\end{definition}
\begin{definition}\label{Banach}(Banach contraction principle) Consider the contraction mapping $f$ on complete metric space $X$ then there exist unique fixed point $\bar{x} \in X$ of $f$.
\end{definition}
\begin{example} Let $(\mathbb{R}, d)$ be a metric space,  where $d$ be a Euclidean distance metric that is, $d(a, b) =|a-b|$. The mapping $f:\mathbb{R}\to \mathbb{R}$  defined as $f(x) = \frac{x}{2}+k$, clearly it is a contraction mapping. Hence, definition \ref{Banach} yields $f$ has a unique fixed point and one can easily see that $2k$ is the unique fixed point of the defined function $f$.
\end{example}
\begin{definition}(Iterated function system {IFS})  Let $f=\{f_1,f_2,\dots ,f_n\}$ is a collection of $n$ contraction
mappings defined on a closed and bounded subset $D \subset \mathbb{R}^n$  that is,  for each $k \in \{1, 2,\dots  n\}, ~ f_{k}: D\to D$. Thus, using the definition (\ref{con})
there exists a constant  $0 \leq \lambda_{k}< 1 $ such that 
$d(f_{k}(a), f_{k}(b)) \leq \lambda_{k} d(a, b)~~ \forall ~~a, b \in D$. Hence, we get an IFS $\{D; f_{i}, i=1, \ldots, k\} $. Let $F: \mathcal{C}(D) \to \mathcal{C}(D)$ be a map on the collection of all the compact subsets of $D$, which is denoted by $\mathcal{C}(D)$.
 For any subset $A\subset D$, $F$ is defined by
$$ F(A)={\bigcup_{i=1}^{n}f_{i}(A)}.$$ The map $F$ follows the definition \ref{Banach} if it is endowed with Hausdorff metric, therefore $F$ has a unique set, say  $A^{*} \subset D$ which is the “fixed point” of $F$ i.e
$$F(A^{*})= A^{*}.$$
Thus, this $A^{*}$ is an attractor of IFS.\\
For example consider following two maps on $[0,1]$:
$f_1(x) =\frac{x}{3}$ , $f_2(x) = \frac{x}{3}+\frac{2}{3}$
clearly we can see Cantor set $C \subset [0, 1]$ is the attractor/fixed point i.e  $C=f_1(C)\cup f_2(C)$ 
\end{definition}
\begin{definition}
Consider $(X, d)$ be a metric space and $P(X)$ denotes the class of all nonempty subsets of $X$ , $K(X)$ denotes  the collection of all nonempty compact subsets of $X$. The function $h : P(X)\times P(X)\to \mathbb{R_+} \cup \{\infty\}$ , $$h(A, B) = \max\{d(A, B), d(B, A)\} = \max\left \{\max_{x\in A}\min_{y\in B}d(x, y), \max_{y\in B}\min_{x\in A}
d(x, y)\right \},$$ 
$\forall A ,B\in P(X)$ is called  Hausdorff semi–metric. When we use $K(X)$ instead of $P(X)$, the previously described mapping $h$ is called Hausdorff metric.

\end{definition}
\begin{lemma}
If $(X, d)$ is complete and compact, then the metric space $(K(X),h)$ is complete and compact respectively.
\end{lemma}
\begin{example}
Let the metric space $\mathbb{R}$ having usual metric $d$ which induced the absolute value, that is 
$d(a,b)=|a-b|,~ a,b \in \mathbb{R} $ . Let $X =[0, 20]$ and  $Y =[22, 31]$. It's worth noting that for each $x \in X$, the closest point in $Y$ that gives the smallest distance will always
be $y = 22$. Therefore, we find that $d(X, Y) = \sup\{d(x, 22) : x \in X\}$. The point $x = 0$ in $X$ maximizes this distance. Therefore $$d(X, y) = d(0, 22) = |22 - 0| = 22.$$
Similarly, we find that $d(Y, X) = \sup\{d(y, 20) : y \in Y\}$, since the point $x = 20$
will give the smallest distance to any point in $Y$. The point $y = 31$ in $Y$ maximizes
this distance, so we have $$d(Y, X) = d(20, 31) = |20-31| = 11.$$  It follows that
$$h(X, Y) = \max\{d_2(X, Y), d_2(Y, X)\} = 22.$$
Note that $d(Y, X)$ and $d(X, Y)$ are not always equal.

\end{example}

\begin{example}Let the metric space $\mathbb{R}$  having usual metric $d$ which induced the absolute value, that is $d(a,b)=|a-b|,~ a,b \in \mathbb{R}$ . Consider 
$A = (0,1]~~ {\mbox{and}}~~ B =[-1,0)$.
Similarly as in previous example we get
$$\displaystyle h(A,B)=1 .$$
\end{example}

\section{Sequences of transformations and trajectories}
For basic terminologies and definitions used in this section, we refer the reader to \cite{DLM,Jha2,JVC1}.
Consider a complete metric space $(X, d)$. Let a sequence of continuous 
transformations $\{T_{i}:X\to X\mid i\in \mathbb{N}\}$.
\begin{definition}(Forward and backward procedures)
Consider $\{T_{i}:X\to X\mid i\in \mathbb{N}\}$ be the sequence of maps, forward and backward procedures are defined as
\begin{itemize}
\item 	$\Phi =T_k\circ T_{k-1}\circ \dots \circ T_{1},$
\item $\Psi= T_1\circ T_2\circ \dots \circ T_k.$
\end{itemize}
\end{definition}
\begin{definition}(Forward and backward trajectories)
We define forward and backward trajectories in $X$ as a result of the forward and backward processes, beginning with $x\in X, \{\Phi_{k}(x)\}$ and $\{\Psi_{k}(x)\}$ are given as
$$\Phi_{k}(x)=T_{k}\circ T_{k-1}\circ \dots \circ T_{1}(x) = T_{k}(x)\circ \Phi_{k-1}(x)  ,k \in \mathbb{N},$$                     
$$\Psi_{k}(x)= T_{1}\circ T_{2}\circ \dots \circ T_{k}(x)= T_{1}(x)\circ \Psi_{k-1}(x)  ,k\in \mathbb{N}.$$
\end{definition}

We will look at the convergence of both forward and backward trajectories in this section. Let us begin by defining the following term before stating our next argument.
\begin{definition} (Asymptotic similarity) Consider $(X,d)$ be a metric space then  
two sequences $\{x_{i}\}_{i\in N}$, $\{y_{i}\}_{i\in N}$ are claimed to be asymptotically similar if $$d(x_{i},y_{i}) \to 0~~ as~~i\to \infty.$$ This relation is denoted by
$$\{x_{i}\}\sim\{y_{i}\}.$$
\end{definition}

\subsection{Some generalized contractions}
Let $(X, d)$ be a complete metric space, now consider $f : X \to X$. We review some contractive-type conditions here. We say that  $\phi: [0,\infty)\to [0,\infty)$ is a comparison function if it is increasing and $\phi^{p}(t) \to 0$ as $p \to \infty$, for every $t \geq 0$,
where $\phi^{p} = \phi^{p}\circ\phi^{p-1}$,  which denotes p-times compositions of function  $\phi$ . One can deduce that, if $\phi$ is a comparison function, then $\phi(t) < t$
for all t $>$ 0, $\phi$ is right continuous at 0 and $\phi(0) = 0.$
$$d(f(x_{1}), f(x_{2})) \leq \phi(d(x_{1}, x_{2})), ~~\forall x_{1}, x_{2} \in X.$$
This mapping $f$ is called  $\phi$ contraction mapping. Then mapping $f$ has  unique fixed point in $X$ i.e $\bar{x} \in X$. Furthermore, for any $x_{0}\in  X$ then the sequence $f_n(x_0)$ always converges to $\bar{x}$.

Obviously, it is a specific case of this outcome, for $\phi(t)= rt,$ where $r\in(0,1)$ of Banach contraction principle,  $\phi(t)= \frac{t}{t+3}$  and $\phi(t) = \ln(t +2)$ are examples of comparison functions. Throughout this paper the notation  $\phi$ is reserved to denote a comparison function.
\begin{definition} If we have a self-map $f$ over a metric space $(X, d)$ for some function $\phi: [0,\infty)\to [0,\infty)$ and we have , 
$$d(f(x_{1}), f(x_{2})) \leq \phi(d(x_{1}, x_{2})), ~~\forall x_{1}, x_{2} \in X.$$
then we say that $f$ is a $\phi$-contraction and  if $\phi$ is non-decreasing mapping  such that
$\phi^{p}(t) \to 0$ as $p \to \infty$, for all $t > 0$, then $f$ is called a Matkowski contraction and if $\alpha(t) =\frac{\phi(t)}{t}< 1$ for any $t > 0$, the function $\alpha$ is non-increasing then, at that point we consider such a mapping a Rakotch contraction .
\end{definition}
\begin{example} Consider $X = [0,\infty)$ and $d(x,y)=|x-y|$ . Let $f_1 : X \to [0, 1]$  and $f_2 : X \to [0, 1]$,   where $[0,1] \subset X$
and  defined $f_1(x) =\frac{1}{1+x}$ and $f_2(x) = \frac{x}{1+x}$. Then $$d(f_1(x),f_1(y))=\frac{|x-y|}{(x+1)(y+1)}\leq \frac{|x-y|}{(|x-y|+1)}=\alpha(d(x,y))d(x,y).$$
$$d(f_1(x),f_1(y)) \leq \alpha(d(x,y))d(x,y).$$
Similarly 
$$d(f_2(x),f_2(y)) \leq \alpha(d(x,y))d(x,y),$$
where $\alpha(t) =\frac{1}{1+t}$ where $t=d(x,y)\geq 0$. Obviously $f_1$ and $f_2$ are Rakotch contraction mappings on $X \subset \mathbb{R}$ over the usual metric but these are not Banach contraction maps.
\end{example}

\begin{proposition} (Asymptotic similarity of trajectories)
Consider sequence of transformations $\{T_{i}\}_{i\in \mathbb{N}}$ on metric space $X$, where every $\{T_{i}\}$ is a $\phi$ contraction mapping along with each comparison mapping $\phi_{i}$. If       
$$\lim_{k \to \infty } \phi_{1}\circ\phi_{2}\circ\phi_{3}\dots \circ\phi_{k} = 0.$$
 then for any $x, y \in X,$
$$\{\Phi_{k}(x)\}\sim \{\Phi_{k}(y)\},$$

$$\{\Psi_{k}(x)\} \sim \{\Psi_{k}(y)\}.$$
\end{proposition}
\begin{proof} Consider $x, y \in X$ and let the trajectories $\{\Psi_{k}(x)\}$ and $\{\Psi_{k}(y)\}$. Utilizing the way that $\{T_{i}\}$ is
a $\phi$ contraction mapping with each has  comparison mapping $\phi_{i}$.
    \begin{align*}
    d(\Psi_{k}(x), \Psi_{k}(y)) & \leq \phi_{1}(d({T_{2}\circ T_{3}\circ \dots \circ T_{k}(x)},{T_{2}\circ T_{3}\circ \dots \circ T_{k}(y)} ) \\ &
\leq \phi_{1}\circ\phi_{2}(d({T_{3}\circ T_{4}\circ \dots  \circ T_{k}(x)},{T_{3}\circ T_{4}\circ \dots  \circ T_{k}(y)} )\\ & 
 \leq \phi_{1}\circ\phi_{2}\circ\phi_{3}(d({T_{4}\circ \dots \circ T_{k}(x)},{T_{4}\circ  \dots \circ T_{k}(y)})\\ &
 \leq \phi_{1}\circ\phi_{2}\circ\phi_{3}\dots \circ\phi_{k}(d((x,y)) \to 0  ~ ~ ~ ~  \text{as}~ k \to \infty.
    \end{align*}

 By using property of $\phi$ mapping that is, $\lim_{k \to \infty } \phi_{1}\circ\phi_{2}\circ\phi_{3} \dots \circ\phi_{k} = 0.$
 \end{proof}
 \begin{remark}
 The proof is alike for the backward and the forward trajectories.
 \end{remark}
 \begin{remark} The condition $\lim_{k \to \infty} \phi_{1}\circ\phi_{2}\circ\phi_{3} \dots \circ\phi_{k} = 0$ condition expressed in proposition 3.6 doesn't guarantee convergence of the trajectories $\{\Phi_{k}(x)\}$.
 \end{remark}
If $T_{i} =T ~ \forall ~ i \in \mathbb{N}$, and $T$ is
a $\phi$ contraction mapping with comparison mapping $\phi$ such that $\phi(t) <$ t for all t $>$ 0 then, at that point both kind  of trajectories are only the fixed-point iteration trajectories ${T^k(x)}$, where $T^k$ is the k-times composition of $T$ itself and  $T^k$ converge to  unique limit for taking any starting point $x\in X$. Which is also known from the principle of Banach contraction (see \cite{PV1})  that ${T^k(x)}$ converges to  unique limit for any $x\in X$. The query now emerges with regard to the convergence of
general trajectories that is, which conditions ensure the convergence of the backward and the forward trajectories.\\
\begin{definition}(Invariant domain)          Any subset of $X$ i.e consider $C \subseteq X$ is called an invariant domain of sequence of  transformations
$\{T_{i}\}_{i\in N}$ if
$\text{take any}~~ x \in C,\text{then}~ T_{i}(x) \in C$, $\forall~ i \in \mathbb{N}$.
\end{definition}
\begin{proposition}(Convergence of forward trajectories)
Consider $\{T_{i}\}_{i\in \mathbb{N}}$ be a sequence of transformations defined on $X$, also having compact
invariant domain $C$, and consider $\{T_{i}\}_{i\in \mathbb{N}}$ converges uniformly  to a
map  $T$  on $C$, where $T$ is a $\phi$ contraction mapping with comparison mapping $\phi$. Then, for taking any 
$x \in C$, the
trajectory $\{\Phi_{i}(x)\}_{i\in \mathbb{N}}$ will  converges to  fixed-point $p$ of $T$ i.e,

$$\lim_{k \to \infty } d(\Phi_{k}(x),p) = 0.$$
\end{proposition}
\begin{proof} Denoting  $~\epsilon_{i} = \sup_{x\in C} d(T_{i}(x), T(x)), i \in \mathbb{N}$, it follows that 

$$\lim_{k \to \infty } \epsilon_{i} = 0.$$
As $T$ is  $\phi$ contraction mapping with comparison mapping $\phi$ such that $\phi(t) < t$ for all $t > 0$, the fixed-point iterations $T^k(x)$ converge to $p\in X$ which is unique fixed-point  for taking any initial element $x$. It additionally obey that $C$ is an invariant domain of $T$. Beginning with taking any point $x \in C$, which gives  that $\Phi_k(x) \subseteq C$. Utilize the triangle inequality of metric space $(X, d)$ and  $\phi$ contraction property of $T$, we get
    \begin{align*}
    & d(\Phi_{k+m}(x),T^m \Phi_{k}(x))  \leq (d({T_{k+m}\circ T_{k+m-1}\circ \dots \circ T_{k+1}\circ \Phi_{k}(x)},T^m \Phi_{k}(x))\\&
    \leq (d({T_{k+m}\circ T_{k+m-1}\circ \dots \circ T_{k+1}\circ \Phi_{k}(x)}, T \circ T_{k+m-1}\circ \dots \circ T_{k+1}\circ \Phi_{k}(x))\\&+ (d( T \circ T_{k+m-1}\circ \dots \circ T_{k+1}\circ \Phi_{k}(x), T^2 \circ T_{k+m-2}\circ \dots \circ T_{k+1}\circ \Phi_{k}(x))\\&+ (d(T^2 \circ T_{k+m-2}\circ \dots \circ T_{k+1}\circ \Phi_{k}(x), T^3 \circ T_{k+m-3}\circ \dots \circ T_{k+1}\circ \Phi_{k}(x))\\&+\dots+(d(T^{m-1} \circ T_{k+1}\circ \Phi_{k}(x), T^m \Phi_{k}(x))\\&
    \leq  \epsilon_{k+m} + \phi \epsilon_{k+m-1} +\phi^2 \epsilon_{k+m-2} + \dots + \phi^{m-1} \epsilon_{k+1}\\&
    \leq  \max_{1\leq i \leq m} \{\epsilon_{k+i}\} \times \frac{1}{1-\phi}.
 \end{align*}

Now utilize the relation  
$$d(\Phi_{k+m}(x),p) \leq d(\Phi_{k+m}(x),T^m \Phi_{k}(x)) +d(T^m \Phi_{k}(x),p),$$ 

the outcome follows by seeing that for $k$ large enough $ \max_{1\leq i \leq m} \{\epsilon_{k+i}\}$ can constructed
 as small as we need and for such $k$, for its sufficient large value , 
$d(T^m \Phi_{k}(x),p)$ is as small as required.
\end{proof}
\begin{proposition} \label{CBT}(Convergence of backward trajectories)
Consider $\{T_{i}\}_{i\in N}$ be a sequence of transformations defined on $X$, also $d(T(x_0),x_0)$ is bounded for some $x_0 \in X$, and assume $T_{i}$ is a $\phi$- contraction mapping with comparison mapping $\phi$ such that $\phi(t) < t$ for all $t > 0$, $$\sum_{k=1}^{\infty} \phi_{1}\circ\phi_{2}\circ\phi_{3} \dots \circ\phi_{k}(x) < \infty,$$ then the backward trajectories $\Psi_{k}(x)= T_{1}\circ T_{2}\circ \dots \circ T_{k}(x) ,k \in N$,
converge for taking any initial element  $x \in X$ to a unique limit  point in $X$.
\end{proposition}
\begin{proof}
 We have   $d(T(x_0),x_0)$ is bounded for some $x_0 \in X$,
that is  $$\sup_{1\leq i < \infty} d(T_{i}(x_0), x_0)) \leq M < \infty, $$
$d(T_{i}(x_0), x_0)) \leq M $  then
$\phi_{i}(d(T_{i}(x_0), x_0)) \leq \phi_{i}(M) $
because $\phi_{i}$ is increasing function. We know $\{\Psi_{k}(x)\}\sim \{\Psi_{k}(y)\}$ for any $x,y \in X$ that is we need to show  $\Psi_k(x_0)$ is convergent  then $\Psi_k(x)$ it is convergent for any $x \in X$  by asymptotic similarity.
    \begin{align*}
        d(\Psi_{k+1}(x_0),\Psi_{k}(x_0))& \leq d(\Psi_{k}(T_{k+1}(x_0), \Psi_{k}(x_0))\\&
\leq \phi_{1}\circ\phi_{2}\circ\phi_{3} \dots  \circ\phi_{k}(d(T_{k+1}(x_0), x_0)) \to 0 ~as~  k \to \infty.
    \end{align*}
For taking  $m,k \in \mathbb{N}, m > k$, we get 
    \begin{align*}
& d(\Psi_{m}(x_0), \Psi_{k}(x_0)) \\&  \leq d(\Psi_{m}(x_0),\Psi_{m-1}(x_0))+ \dots +d(\Psi_{k+2}(x_0), \Psi_{k+1}(x_0))+d(\Psi_{k+1}(x_0),\Psi_{k}(x_0))\\
& \leq \phi_{1}\circ \dots \circ\phi_{m-1} d(T_{m}(x_0), x_0)+\dots+ \phi_{1}\circ\dots\circ\phi_{k+1} (d(T_{k+2}(x_0), x_0))\\& + \phi_{1}\circ\dots\circ\phi_{k} (d(T_{k+1}(x_0), x_0)\\
& \leq \phi_{1}\circ\dots\circ\phi_{m-1}(M)+\phi_{1}\circ\dots\circ\phi_{k+1}(M)+ \phi_{1}\circ\dots\circ\phi_{k}(M)\\&
=S_{m-1}(M)-S_{k}(M),
\end{align*}
where $S_k(t)=\sum_{k=1}^{\infty} \phi_{1}\circ\phi_{2}\circ\phi_{3}\dots\circ\phi_{k}(t) $  for any $t>0$.
Since $\sum_{k=1}^{\infty} \phi_{1}\circ\phi_{2}\circ\phi_{3}\dots\circ\phi_{k}(t) < \infty $ for any $t>0$, it follows that $S_{m-1}(M)-S_{k}(M) \to 0$. From this, we get  
$$d(\Psi_{m}(x_0), \Psi_{k}(x_0))  \to 0 ~\text{as}~  k \to \infty.$$
That is,
$\{\Psi_{k}(x_0)\}_{k\in \mathbb{N}}$  becomes a  Cauchy sequence and because of  completeness property of $(X, d)$, this is convergent
for some  $x_0 \in X$. The uniqueness of the limit point is attain by the equivalence property of all trajectories.
\end{proof}

\begin{remark}(Differences between forward and backward trajectories)
1) We observe that if $T_{i} \to T$ and $T$ is a $\phi$ contraction mapping with comparison mapping $\phi$ such that $\phi(t) < t$ for all $t> 0$,             $$\sum_{k=1}^{\infty}  \phi_{1}\circ\phi_{2}\circ\phi_{3}\dots \circ\phi_{k} <\infty$$ 
and both trajectories, the backward and the forward  trajectories converge.\\\\
2) The sufficient condition for the asymptotic similarity result for the both backward and forward trajectories is  $$\lim_{k \to \infty } \phi_{1}\circ\phi_{2}\circ\phi_{3}\dots \circ\phi_{k} = 0.$$ Under the more severe conditions 
$$\sum_{k=1}^{\infty} \phi_{1}\circ\phi_{2}\circ\phi_{3}\dots \circ\phi_{k} < \infty$$ and the existence of some $x_0\in X$ such that $d(T(x_0),x_0)$ is bounded then for the backward trajectories, we achieve convergence.\\\\
3) In numerous cases the forward trajectories don't converge while the backward trajectories converge. In order to explain this, consider the metric space $R$ with distance function define as $d(x, y) = |x - y|$ and let us consider the simple sequence of contractive transformations $\phi_{2i-1}(x)$ =$ x/2$, $\phi_{2i} = x/2 + c$, $i \geq 1$ . The backward trajectories converge to the fixed point of $\Phi_{1} = \phi_{1} \circ \phi_{2}$, which is $2c/3$. The forward trajectories have two accumulation points, which are the fixed point of $\Phi_{1}$ i.e., $2c/3$, and the fixed point of $\Psi_{2} = \phi_{2} \circ \phi_{1}$
which is $4c/3$.
\end{remark}
\section{Trajectories of sequences of function systems}
Classical generalisation of IFS, let us  consider a sequence of function systems (SFS), as well as its trajectories. Assume that $(X, d)$ be a complete metric space. Let an SFS $\{F_{i}\}_{i\in N}$ characterised by
$F_{i} = \{X; f_{1,i}, f_{2,i},\dots,f_{ni,i}\},$

where $f_{r,i} : X \to X$ all are continuous maps. The set-valued maps that go with it are given as
$F_{i} : H(X) \to H(X)$ and
$$ F_{i}(A)={\bigcup _{r=1}^{n_{i}}f_{r,i}(A)},$$ 
where denoting $f_{r,i} ~ \phi$-contraction  mapping with comparison mapping $\phi_{r,i}$ such that $\phi_{r,i}(t) < t$ for all $t > 0$ and for $r = 1, 2,\dots,n_{i}$. We recall that ,
the comparison mapping of $F_{i}$ in $(H(X),h)$ is  $max_{r=1,2,\dots,n_{i}} \phi_{r,i} = \phi_{i}$.
The attractor is addressed in classical IFS theory  apparently, the set that is a map $F_i$'s fixed-point. In this section, we assume the trajectories
of sequence of function system (SFS) maps $\{F_{i}\}_{i\in N}$, we call it forward and backward SFS trajectories\\
$$\Phi_{k}(A) ={F_{k}\circ F_{k-1}\circ \dots \circ F_{1}}(A) , k \in N,$$ 
$$\Psi_{k}(A)= {F_{1}\circ F_{2}\circ \dots \circ F_{k}}(A)  , k \in N$$  
respectively.
$H(X)$, possessing the Hausdorff metric $h$ and will be complete metric space if given that $(X, d)$ is complete metric space.\\\\
\begin{corollary}(Asymptotic similarity of SFS trajectories)
Let an SFS characterised by $F_{i} = \{X; f_{1,i}, f_{2,i},\dots,f_{ni,i}\}, i \in N$, here $f_{r,i} : X \to X$ are ${\phi}$-contraction mapping having comparison mapping $\phi_{r,i}$. Further, assume the set valued maps $F_{i}$ on $(H(X), h)$ with the corresponding comparison mapping $\phi_{i}$ satisfy for any $t>0,$
 $$\lim_{k \to \infty } \phi_{1}\circ\phi_{2}\circ\phi_{3}\dots\circ\phi_{k}(t) = 0.$$  As a result, all of $F_{i}$'s forward trajectories  are asymptotically similar and all of $F_{i}$'s backward trajectories are asymptotically similar that is 
 $\forall A ,B \in H(X)$
 $$\{\Phi_{k}(A)\}\sim \{\Phi_{k}(B)\},$$
$$\{\Psi_{k}(A)\}\sim \{\Psi_{k}(B)\}.$$
\end{corollary}

\begin{remark}
For both backward and forward trajectories, the proof is the same.
\end{remark}
\begin{corollary}(Convergence of forward SFS trajectories)\\ Consider $\{F_{i}\}_{i\in N}$ as in the corollary 4.1, there are an equal number of mappings, $n_{i} =n $ and consider $F_{i} = \{X; f_{1}, f_{2},\dots,f_{n}\}, ~i \in N$. Suppose $\exists$   $C \subseteq X$, which is compact invariant domain for ${f_{r,i}}$ as well as for each $r = 1, 2,\dots,n$ and the $\{f_{r,i}\}_{i\in N}$ sequence
converges to $f_r$ uniformly over $C$ as i $\to$ $\infty$. Also suppose that $F$ is ${\phi}$-contraction mapping then the forward trajectories ${\Phi_{k}(A)}$ converge for taking any initial subset of $c$ i.e $P \subseteq C$ to a unique attractor of $F$.
\end{corollary}

\begin{corollary}(Convergence of backward SFS trajectories)\\
Let an SFS   $F_{i} = \{X; f_{1}, f_{2},\dots,f_{n}\},$ $ i \in N$ defined as in corollary 4.1 and $d(F_i(A_1),A_1)$ is bounded for some $A_1 \subseteq H(X)$ assume that for any $t>0$ and  $$\sum_{k=1}^{\infty} \phi_{1}\circ\phi_{2}\circ\phi_{3}\dots\circ\phi_{k}(t) < \infty,$$ then the backward trajectories $\Psi_{k}(A)= F_{1}\circ F_{2}\circ \dots \circ F_{k}(A) ~,k \in N $ 
converge for any $P \subseteq C$ is unique set (attractor) in $C$.
\end{corollary}
\begin{lemma}
Consider $(X, d)$ be a metric space, $\phi : [0,+\infty) \to [0,+\infty)$ be non-decreasing function and $f$ be a $\phi$-contractive. Then $h_d(F(A), F(B)) \leq \phi (h_d (A, B))$ for all $A, B \in H(X)$.
That is, $F : H(X) \to H(X)$ is also a $\phi$-contractive (where $\phi$ is comparison function), where 
$\forall D ~\in H(X), ~~~  F(D) = f(D)$.
\end{lemma}
\begin{proof}
The distance function $d(x, y)$ for fixed $x \in X$ is continuous in $y \in X$.
Let $A, B \in H(X)$ and set any $x_0 \in A$.
Using the compactness of $A$, there is $y_{x_0} \in B$ such that $$  \min_{y\in B} d(x_0, y) = d(x_0, y_{x_0}),$$ then we obtain $$ \inf_{y\in B}
\phi(d(x_0, y)) \leq \phi(d(x_0, y_{x_0})) = \phi (\min_{y\in B}d(x_0, y)),$$
because $\phi : [0,+\infty) \to [0,+\infty)$ is non-decreasing, it follows that
$$\phi(\min_{y\in B}d(x_0, y))\leq \phi(\max_{x\in A} \min_{y\in B}d(x, y)) \leq \phi(h_d (A, B)).$$
Since $x_0$ was arbitrary, $$\sup_{x\in A} \phi(\min_{y\in B} d(x, y)) \leq \phi(h_d(A, B)).$$
Hence, $$\sup_{x\in A}
\inf_{y\in B} \phi(d(x, y)) \leq \sup_{x\in A} \phi(\min_{y \in B} d(x, y)) \leq \phi(h_d(A, B)).$$
Similarly,$$\sup_{y\in B}
\inf_{x\in A} \phi(d(x, y)) \leq \phi(h_d(A, B)).$$
So  
    \begin{align*}
d(F(A),F(B)) & = \max_{f (x) \in F(A)}\min_{f(y)\in F(B)}d(f(x),f(y)) \\& = \max_{x\in A}\min_{y\in B}d(f(x),f(y)) \\ & \leq
\sup_{x\in A}
\inf_{y\in B} \phi(d(x, y)) \\ &\leq \phi(h_d(A, B)).
    \end{align*}
Similarly,
$$d(F(A),F(B))\leq \sup_{x\in A}
\inf_{y\in B} \phi(d(x, y))  \leq \phi(h_d(A, B)).$$
Hence, we have $$h_d(F(A),F(B)) = \max\{d(F(A),F(B)),d(F(B),F(A))\} \leq \phi(h_d(A,B)),$$ this completing the proof.
\end{proof}

In the next theorem, we show an application of the non-stationary fixed point results established previously to fractal interpolation theory. 
\begin{theorem}
Let $\{(x_i,y_i): i = 0, 1,\dots , N\}$ be a data set such that $x_0< x_1< \dots< x_N$ and $I:=[x_0,x_N]$. We consider $l_i: I \to [x_{i-1},x_i].$
Assume that maps $F_{i,k}$ are Matkowski contractions (with the same function $\phi$) with respect to the second variable, that is for some non-decreasing function $\phi : \mathbb{R_+} \to \mathbb{R_+}$ with $\phi^n(t) \to 0$ for $t > 0$ each map $F_i$ satisfies $$|F_{i,k}(x, y) - F_{i,k}(x, y')| \leq \phi(|y - y'|) ~~~\forall x\in I, ~~\forall y,y'\in [a,b]$$
Then the operator $T_k: C^*(I) \to C^*(I))$ defined by $$ (T_kf)(x)=F_{i,k}(l_i^{-1}(x), f(l_i^{-1}(x)), $$ is a Matkowski contraction.
Then there is a unique continuous function $f^*:I\to[a, b]$ which is the limit of the backword trajectories $\Psi_{k}(g)= T_{1}\circ T_{2}\circ \dots \circ T_{k}(g) $ for any $g \in  C^*(I)$ satisfying $f^*(x_i) = y_i$ for $i = 0, 1,\dots , N$.
\end{theorem}
\begin{proof}
 For all $g,h \in C^*(I),$ we have
    \begin{align*}
    d_{C(I)}(T_kg,T_kh)& =\max_{x\in[x_0,x_N]}|T_kg(x)-T_kh(x)|\\&
    =\max_{i=1,2,...,N}\max_{x\in[x_{i-1},x_i]}|T_k g(x) - T_k h(x)|\\&
    = \max_{i=1,2,...,N}\max_{x\in[x_{i-1},x_i]}
    |F_{i,k}(l_i^{-1}(x), g(l_i^{-1})(x))-F_{i,k}(l_i^{-1}(x), h(l_i^{-1})(x))| \\&
    \leq \max_{i=1,2,...,N}\sup_{x\in[x_{i-1},x_i]} \phi(|g({l_i}^{-1}(x)-h({l_i}^{-1}(x))|).
    \end{align*}
 
    Since $\phi : (0,+\infty)\to (0,+\infty)$ is non-decreasing function and ${l_i}^{-1}:[x_{i-1}, x_{i}] \to [x_0,x_N]$ for all $i = 1, 2, . . . , N$, we obtain that for $i_0 \in {1, 2, . . . , N}$ and $x_* \in [x_{i_0-1},x_{i_0}]$,
    \begin{align*}
        \phi(|g({l_{i_0}}^{-1}(x_*)-h({l_{i_0}}^{-1}(x_*))|)&\leq \phi( \max_{x\in[x_{i_0-1},x_{i_0}]}|g({l_{i_0}}^{-1}(x)-h({l_{i_0}}^{-1}(x))|)\\&
    \leq \phi(\max_{x\in[x_0,x_N]}|g(x)-h(x)|\\&
    =\phi(d_{C(I)}(g,h)).
    \end{align*}
Since $x_*$ was arbitrary, therefore
$$\sup_{x\in [x_{i_0-1},x_{i_0}]} \phi(|g({l_{i_0}}^{-1}(x))-h({l_{i_0}}^{-1}(x))|)\leq
\phi(d_{C(I)}(g,h)),$$
and since $i_0$ was arbitrary
$$\max_{i=1,2,...,N}\sup_{x\in[x_{i−1},x_i]}\phi(|g({l_i}^{-1}(x)-h({l_i}^{-1}(x))|)\leq\phi(d_{C(I)}(g,h)).$$
Hence, we obtain
  \begin{align*}
      d_{C(I)}(T_kg,T_kh) &\leq
\max_{i=1,2,...,N}\sup_{x\in[x_{i-1},x_i]} \phi(|g({l_i}^{-1}(x)-h({l_i}^{-1}(x))|)\\&\leq\phi(d_{C(I)}(g,h)).
  \end{align*}  
So we conclude that $T_k$ is a Matkowski $\phi$-contractive map on the complete metric space $(C^*(I),d_{C(I)})$. It is easy to check that all hypotheses of Proposition \ref{CBT} are satisfied. Now, by using Proposition \ref{CBT}, we have a unique function $f^*:I\to[a, b]$ satisfying $f^*(x_i) = y_i$ for $i = 0, 1, . . . , N$, which is the limit of the backward trajectories $\Psi_{k}(g)= T_{1}\circ T_{2}\circ \dots \circ T_{k}(g) $ for any $g \in  C^*(I).$ This completes the proof.
\end{proof}
\begin{remark}
The above result generalizes several results available in the fractal interpolation theory, see, for instance, \cite{MB, SR1,SR2,sverma5}.
\end{remark}

\section{ Countable Iterated Function System} 
A sequence of continuous maps $(f_n)_{n\geq 1}$ on a topological space
$X$ into itself is called a Countable Iterated Function System (CIFS) on X. The associated set function $F : P(X)\to P(X)$ is defined by
$$ F(A)=\overline{{\bigcup_{i=1}^{\infty}f_{i}(A)}} ,~~~\forall A \in H(X).$$
A set $A \in P(X)$ for which $F(A) = A$, is named set fixed point (s.f.p.) of the CIFS. The set $A$ is called attractor of the CIFS whenever $A \in K(X)$  ( where $K(X)$ is the class of all compact sets of $P(X)$) and it is the unique s.f.p. for the respective CIFS.

\begin{theorem}
Let $(X, d)$ be a metric space, $\phi : [0,+\infty) \to [0,+\infty)$ be non-decreasing function and $f$ be a $\phi$-contractive. Then $h_d(F(A), F(B)) \leq \phi (h_d (A, B))$ for all $A, B \in H(X)$.
That is, $F: H(X) \to H(X)$ is also a $\phi$-contractive (where $\phi$ is comparison function), where
$$ F(A)=\overline{{\bigcup_{i=1}^{\infty}f_{i}(A)}} ,~~~~~\forall A \in H(X).$$
\end{theorem}
\begin{proof}
The distance function $d(x, y)$ for fixed $x \in X$ is continuous in $y \in X$.\\
Let $A, B \in H(X)$ and set any $x_0 \in A$.
Using the compactness of $A$, there is $y_{x_0} \in B$ such that $$  \min_{y\in B} d(x_0, y) = d(x_0, y_{x_0}),$$ then we obtain $$ \inf_{y\in B}
\phi(d(x_0, y)) \leq \phi(d(x_0, y_{x_0})) = \phi (\min_{y\in B}d(x_0, y)),$$
because $\phi : [0,+\infty) \to [0,+\infty)$ is non-decreasing, it follows that
$$\phi(\min_{y\in B}d(x_0, y))\leq \phi(\max_{x\in A} \min_{y\in B}d(x, y)) \leq \phi(h_d (A, B)).$$
Since $x_0$ was arbitrary, $$\sup_{x\in A} \phi(\min_{y\in B} d(x, y)) \leq \phi(h_d(A, B)).$$
Hence, $$\sup_{x\in A}
\inf_{y\in B} \phi(d(x, y)) \leq \sup_{x\in A} \phi(\min_{y \in B} d(x, y)) \leq \phi(h_d(A, B)).$$
Similarly,$$\sup_{y\in B}
\inf_{x\in A} \phi(d(x, y)) \leq \phi(h_d(A, B)).$$
So  
    \begin{align*}
d(F(A),F(B)) & =d\Big(\overline{\bigcup_{i=1}^{\infty}f_{i}(A)},\overline{\bigcup_{i=1}^{\infty}f_{i}(B)}\Big)\\& \leq \sup_{i\in N} d(f_{i}(A),f_{i}(B))\\& = \sup_{i\in N} \{\max_{f_i (x) \in f_i(A)}\min_{f_i(y)\in f_i(B)}d(f_i(x),f_i(y))\} \\& = \sup_{i\in N}\{\max_{x\in A}\min_{y\in B}d(f_i(x),f_i(y))\} \\ & \leq {\sup_{x\in A}
\inf_{y\in B} \phi(d(x, y))} \\ &\leq \phi(h_d(A, B)).
    \end{align*}
Similarly,
$$d(F(A),F(B))\leq \sup_{x\in A}
\inf_{y\in B} \phi(d(x, y))  \leq \phi(h_d(A, B)).$$
Hence, we get $$h_d(F(A),F(B)) = \max\{d(F(A),F(B)),d(F(B),F(A))\} \leq \phi(h_d(A,B)),$$ establishing the claim.
\end{proof}
\section{Acknowledgements}
 The work of first author  is financially supported by the
CSIR, India with grant no:
 09/1028(0019)/2020-EMR-I.

\bibliographystyle{amsplain}

\end{document}